\documentclass{article}
\usepackage{amsmath,amssymb,amsthm}
\usepackage{graphicx,subfigure,float,url,color}
\usepackage{mathrsfs}
\usepackage[colorlinks=true]{hyperref}
\usepackage{pdfsync}

\graphicspath{{fig/}}

\def\R{\textrm{I\kern-0.21emR}}
\def\N{\textrm{I\kern-0.21emN}}

\newcommand{\C} {\mathbb{C}}

\renewcommand{\geq}{\geqslant}
\renewcommand{\leq}{\leqslant}

\newtheorem{theorem}{Theorem}
\newtheorem{proposition}{Proposition}

\newtheorem{lemma}{Lemma}

\theoremstyle{definition}\newtheorem{remark}{Remark}

\begin{document}

\title{Characterization by observability inequalities of controllability and stabilization properties}

\author{Emmanuel Tr\'elat\thanks{
Sorbonne Universit\'e, Universit\'e Paris-Diderot SPC, CNRS, Inria, Laboratoire Jacques-Louis Lions, \'equipe CAGE, F-75005 Paris (\texttt{emmanuel.trelat@sorbonne-universite.fr}).} \quad
Gengsheng Wang\thanks{
Center for Application Mathematics, Tianjin University, Tianjin, 300072, China (wanggs@yeah.net)}\quad
Yashan Xu\thanks{School of Mathematical Sciences, Fudan University, KLMNS, Shanghai, 200433,
China  (yashanxu@fudan.edu.cn)}}
\date{}
\maketitle

\date{}
\maketitle

\begin{abstract}
Given a linear control system in a Hilbert space with a bounded control operator, we establish a characterization of exponential stabilizability in terms of an observability inequality. Such dual characterizations are well known for exact (null) controllability. Our approach exploits classical Fenchel duality arguments and, in turn, leads to characterizations in terms of  observability inequalities of approximately null controllability and of $\alpha$-null controllability.
We comment on the relationships between those various concepts, at the light of the observability inequalities that characterize them.
\end{abstract}

\section{Introduction and main results}
Let $X$ and $U$ be Hilbert spaces. We consider the linear control system
\begin{equation}\label{contsyst}
\dot y(t) = Ay(t)+Bu(t) \qquad t\geq 0
\end{equation}
where $A:D(A)\rightarrow X$ is a linear operator generating a $C_0$ semigroup $(S(t))_{t\geq 0}$ on $X$ and $B\in L(U,X)$ is a linear bounded control operator.

It is well known that, for such classes of control systems, exact (null) controllability is equivalent by duality to an observability inequality. In the existing results (e.g., heat, wave, Schr\"odinger equations), such inequalities are instrumental to establish controllability properties (see the textbooks \cite{CurtainZwart, LT1, Lions_HUM, Staffans, TW, Z}). Besides, exponential stabilizability, meaning that there exists a feedback operator $K$ such that $A+BK$ generates an exponentially stable semigroup, is characterized in the existing literature in terms of infinite-horizon linear quadratic optimal control and algebraic Riccati theory (see the previous references). But, up to our knowledge, a dual characterization of exponential stabilizability in terms of an observability inequality is not known.

In this paper, we mainly prove that the control system \eqref{contsyst} is exponentially stabilizable if and only if for every $\alpha\in(0,1)$ (equivalently, there exists $\alpha\in(0,1)$) there exist $T>0$ and $C>0$ such that
\begin{equation}\label{weakobs}
\Vert S(T)^*\psi\Vert_{X}
\leq C \Vert B^*S(T-\,\cdot\,)^*\psi\Vert_{L^2(0,T;U)} + \alpha\Vert\psi\Vert_X \qquad\forall\psi\in X .
\end{equation}
Moreover, the best stabilization decay rate is the infimum of $\frac{\ln\alpha}{T}$ over all possible couples $(\alpha,T)$ such that the above observability inequality is satisfied for some $C<+\infty$ (see Theorem \ref{thm_stab} hereafter for a more complete statement). Besides, we  present some characterizations, in terms of  observability inequalities, for  approximate null controllability and $\alpha$-null controllability (which will be introduced later) and give some connections among these concepts.

Our proof follows very simple arguments essentially exploiting Fenchel duality.
We insist that, in the weak observability inequality \eqref{weakobs}, the coefficient $\alpha$ satisfies $0<\alpha<1$. The limit case $\alpha=1$ is critical.

The above-mentioned observability inequalities involve some constants that quantify those various concepts of controllability or of stabilizability and thus open new issues of investigation.
They also shed some light on implications between those concepts.

\subsection{Statement of the main results}
To state our main results in more details, we start with several reminders and notations.

Given an initial state $y_0\in X$ and a control $u\in L^2_{\mathrm{loc}}(0,+\infty;U)$, the unique solution $y(t)=y(t;y_0,u)$ ($t\geq 0$) to \eqref{contsyst}, associated with $u$ and the initial condition: $y(0)=y_0$, satisfies
$$
y(T;y_0,u) = S(T)y_0 + L_T u \qquad\forall T>0
$$
with $L_T\in L(L^2(0,T;U),X)$ defined by $L_Tu=\int_0^T S(T-t)Bu(t)\, dt$.
Note that $y\in C^0([0,+\infty),X)\cap H^1_{\mathrm{loc}}(0,+\infty;X_{-1})$, with $X_{-1}=D(A^*)'$, the dual of $D(A^*)$ with respect to the pivot space $X$ (see, e.g., \cite{EN,TW}).
We recall that the dual mapping $L_T^*\in L(X,L^2(0,T;U))$ is given by $(L_T^*\psi)(t)=B^*S(T-t)^*\psi$ for every $\psi\in X$.

Throughout the paper we identify $X$ (resp., $U$) with its dual $X'$ (resp., $U'$). We denote by $\Vert\cdot\Vert_X$ and $\langle\cdot,\cdot\rangle_X$ (resp. $\Vert\cdot\Vert_U$ and $\langle\cdot,\cdot\rangle_U$) the Hilbert norm and scalar product in $X$ (resp., in $U$).

\medskip

Hereafter, we denote $\Vert B^*S(T-\,\cdot\,)^*\psi\Vert_{L^2(0,T;U)} = \left(\int_0^T \Vert B^*S(T-t)^*\psi\Vert_{U}^2\, dt\right)^{1/2}$.
We recall that:
\begin{itemize}
\item The control system \eqref{contsyst} is exactly controllable in time $T>0$ if, for all $y_0,y_1\in X$, there exists $u\in L^2(0,T;U)$ such that $y(T;y_0,u)=y_1$.

Equivalently, $\mathrm{Ran}(L_T)=X$, or, by duality, $L_T^*$ is bounded below, which means that there exists $C_T>0$ such that
\begin{equation}\label{obs_exactcont}
\Vert\psi\Vert_{X}\leq C_T \Vert B^*S(T-\,\cdot\,)^*\psi\Vert_{L^2(0,T;U)} \qquad \forall\psi\in X
\end{equation}
(observability inequality). In terms of the (symmetric positive semidefinite) Gramian operator $G_T\in L(X)$ defined by
$$
G_T = \int_0^T S(T-t)BB^*S(T-t)^*\, dt  ,
$$
noting that $\langle G_T\psi,\psi\rangle_X=\Vert G_T^{1/2}\psi\Vert_X^2=\int_0^T \Vert B^*S(T-t)^*\psi\Vert_{U}^2\, dt$, \color{blue}\eqref{obs_exactcont}
\color{black} is also equivalent to the inequality
$$
G_T \geq \frac{1}{C_T^2}\mathrm{id}
$$
written in terms of positive selfadjoint operators.

\item The control system \eqref{contsyst} is exactly null controllable in time $T>0$ if, for every $y_0\in X$, there exists $u\in L^2(0,T;U)$ such that $y(T;y_0,u)=0$.

Equivalently, $\mathrm{Ran}(S(T))\subset\mathrm{Ran}(L_T)$, or, by duality, there exists $C_T>0$ such that
\begin{equation}\label{obs_null}
\Vert S(T)^*\psi\Vert_{X} \leq C_T \Vert B^*S(T-\,\cdot\,)^*\psi\Vert_{L^2(0,T;U)}  \qquad \forall\psi\in X
\end{equation}
(observability inequality), i.e.,
$$
G_T \geq \frac{1}{C_T^2}S(T)S(T)^*.
$$

\item The control system \eqref{contsyst} is approximately controllable in time $T>0$ if, for every $\varepsilon>0$, for all $y_0,y_1\in X$, there exists $u\in L^2(0,T;U)$ such that $\Vert y(T;y_0,u)-y_1\Vert_X\leq\varepsilon$.

Equivalently, $\mathrm{Ran}(L_T)$ is dense in $X$, or, by duality, $L_T^*$ is injective, which means that, given any $\psi\in X$, if $B^*S(T-t)^*\psi=0$ for every $t\in[0,T]$ then $\psi=0$ (unique continuation property).

\item The control system \eqref{contsyst} is approximately null controllable in time $T>0$ if, for every $\varepsilon>0$, for every $y_0\in X$, there exists $u\in L^2(0,T;U)$ such that $\Vert y(T;y_0,u)\Vert_X\leq\varepsilon$.

Equivalently, $\mathrm{Ran}(S(T))$ (or its closure) is contained in the closure of $\mathrm{Ran}(L_T)$, or, by duality, 
given any $\psi\in X$, if $B^*S(T-t)^*\psi=0$ for every $t\in[0,T]$ then $S(T)^*\psi=0$.

Approximate controllability in time $T$ is equivalent to approximate null controllability in time $T$ when $S(T)^*$ is injective, i.e., when $\mathrm{Ran}(S(T))$ is dense in $X$.
This is the case when $(S(t))_{t\geq 0}$ is either a group or an analytic $C_0$ semigroup.

\item Given some $\alpha>0$, the control system \eqref{contsyst} is $\alpha$-null controllable in time $T$ if, for every $y_0\in X$, there exists $u\in L^2(0,T;U)$ such that $\Vert y(T;y_0,u)\Vert_X\leq\alpha\Vert y_0\Vert_X$.

Note that the control system \eqref{contsyst} is approximately null controllable in time $T$ if and only if it is $\alpha$-null controllable in time $T$ for every $\alpha>0$.

\item The control system \eqref{contsyst} is exponentially stabilizable if there exists a feedback operator $K\in L(X,U)$ such that $A+BK$ generates an exponentially stable $C_0$ semigroup $(S_K(t))_{t\geq 0}$, i.e., there exists $M\geq 1$ and $\omega<0$ such that
\begin{equation}\label{inegsemigroup}
\Vert S_K(t)\Vert\leq M\mathrm{e}^{\omega t} \qquad\forall t\geq 0 .
\end{equation}
The infimum $\omega_K$ of all possible real number $\omega$ such that \eqref{inegsemigroup} is satisfied for some $M\geq 1$ is the growth bound of the semigroup $(S_K(t))_{t\geq 0}$ and is given (see \cite{EN,Pazy}) by
\begin{equation*}
\omega_K = \inf_{t>0} \frac{1}{t} \ln \Vert S_K(t)\Vert_{L(X)} = \lim_{t\rightarrow+\infty} \frac{1}{t} \ln \Vert S_K(t)\Vert_{L(X)} .
\end{equation*}
Exponential stabilizability means that there exists $K\in L(X,U)$ such that $\omega_K<0$.

The best stabilization decay rate is defined by
\begin{equation}\label{defomegastar}
\omega^* = \inf \{ \omega_K\ \mid\ K\in L(X,U) \} .
\end{equation}
When $\omega^*=-\infty$, the control system \eqref{contsyst} is said to be \emph{completely stabilizable}: this means that stabilization can be achieved at any decay rate. We also speak of \emph{rapid stabilization}.

\end{itemize}
In the items above, the equivalence between exact controllability (resp., exact null controllability) with the observability inequality \eqref{obs_exactcont} (resp., \eqref{obs_null}) is well known (see \cite{CurtainZwart, LT1, Lions_HUM, Staffans, TW, Z}).

\medskip

In this paper we establish characterizations in terms of observability inequalities, of the concepts of $\alpha$-null controllability, of approximate controllability and of exponential stabilizability. The two first ones follow from easy arguments using Fenchel duality, as in \cite{GlowinskiLionsHe,Lions_1992} and are certainly already known (at least, we were able to find the one on approximate controllability in \cite{Boyer}, in a slightly different form). The last one, characterizing exponential stabilizability, seems to be new. We will give details on all of them in particular because stressing on the differences of the corresponding observability inequalities provides a new insight on the differences between those various notions of controllability or stabilizability.

Let us now state our main results.

\medskip

As already noted, we have $\Vert G_T^{1/2}\psi\Vert_X=\Vert B^*S(T-\,\cdot\,)^*\psi\Vert_{L^2(0,T;U)}$ for every $\psi\in X$.
We define
\begin{equation}\label{def_muy0alpha}
\mu_{y_0,\alpha}^T =\inf\left\{C\geq  0 \ \mid \
\langle\psi,S(T)y_0\rangle_X-\alpha\Vert y_0\Vert_X\Vert\psi\Vert_X \leq C \Vert G_T^{1/2}\psi\Vert_X
\quad \forall\psi\in X\right\}
\end{equation}
with the convention that $\inf\emptyset=+\infty$.
Actually, when the above set is not empty, the infimum is reached (see Section \ref{sec:proofs}).
By definition, we have $\mu_{y_0,\alpha}^T\in[0,+\infty]$; $\mu_{y_0,\alpha_2}^T\leq\mu_{y_0,\alpha_1}^T$ if $\alpha_1\leq\alpha_2$; $\mu_{y_0,\alpha}^T=0$ if $\alpha\geq\Vert S(T)\Vert_{L(X)}$; $\mu_{\lambda y_0,\alpha}^T=\lambda\mu_{y_0,\alpha}^T$ for every $\lambda>0$ (and thus $\mu_{y_0,\alpha}^T=\mu_{\frac{y_0}{\Vert y_0\Vert_X},\alpha}^T\Vert y_0\Vert_X$).
This homogeneity property leads us to define
\begin{equation}\label{def_muT}
\mu_\alpha^T = \sup_{\Vert y_0\Vert_X=1} \mu_{y_0,\alpha}^T.
\end{equation}
We claim that
\begin{equation}\label{connectionconsts10}
\mu_{\alpha}^T =\inf\left\{C\geq  0 \ \mid \
\Vert S(T)^*\psi\Vert_X-\alpha\Vert\psi\Vert_X \leq C \Vert G_T^{1/2}\psi\Vert_X \quad \forall\psi\in X\right\}
\end{equation}
(see Lemma \ref{appendixprop2} in Section \ref{sec_expstab}) and we have as well $\mu_{\alpha}^T\in[0,+\infty]$, $\mu_{\alpha_2}^T\leq\mu_{\alpha_1}^T$ if $\alpha_1\leq\alpha_2$, and $\mu_{\alpha}^T=0$ if $\alpha\geq\Vert S(T)\Vert_{L(X)}$.

\paragraph{Characterization of $\alpha$-null controllability by an observability inequality.}

\begin{theorem}\label{thm_alpha_null}
Let $T>0$ and $\alpha>0$ be arbitrary. The following items are equivalent:
\begin{itemize}
\item The control system \eqref{contsyst} is $\alpha$-null controllable in time $T>0$.
\item For every $y_0\in X$, we have $\mu_{y_0,\alpha}^T<+\infty$.
\item For every $y_0\in X$, there exists $C>0$ such that
\begin{equation}\label{obs_alpha_null}
\langle\psi,S(T)y_0\rangle_{X}\leq C \Vert B^*S(T-\,\cdot\,)^*\psi\Vert_{L^2(0,T;U)} + \alpha\Vert y_0\Vert_X\Vert\psi\Vert_X  \qquad\forall\psi\in X .
\end{equation}
\end{itemize}
When one of these items is satisfied, the smallest possible constant $C$ in the observability inequality \eqref{obs_alpha_null} is $C=\mu_{y_0,\alpha}^T$.
\end{theorem}

\paragraph{Characterization of approximate null controllability by an observability inequality.}

\begin{theorem}\label{thm_approx}
Let $T>0$ be arbitrary. The following items are equivalent:
\begin{itemize}
\item The control system \eqref{contsyst} is approximately null controllable in time $T>0$.
\item For every $y_0\in X$, for every $\alpha>0$, we have $\mu_{y_0,\alpha}^T<+\infty$.
\item For every $y_0\in X$, for every $\alpha>0$, there exists $C>0$ such that
\begin{equation}\label{obs_approx}
\langle\psi,S(T)y_0\rangle_{X} \leq C \Vert B^*S(T-\,\cdot\,)^*\psi\Vert_{L^2(0,T;U)} + \alpha\Vert y_0\Vert_X\Vert\psi\Vert_X  \qquad\forall\psi\in X .
\end{equation}
\end{itemize}
When one of these items is satisfied, the smallest possible constant $C$ in the observability inequality \eqref{obs_approx} is $C=\mu_{y_0,\alpha}^T$.
\end{theorem}

Equivalently, one can replace ``every $\alpha>0$" with ``every $\alpha\in(0,\Vert S(T)\Vert_{L(X)})$" in Theorem \ref{thm_approx}.

\paragraph{Characterization of exponential stabilizability by an observability inequality.}

\begin{theorem}\label{thm_stab}
The following items are equivalent:
\begin{itemize}
\item The control system \eqref{contsyst} is exponentially stabilizable.
\item For every $\alpha\in(0,1)$ (equivalently, there exists $\alpha\in(0,1)$), there exists $T>0$ such that $\mu_\alpha^T<+\infty$.
\item For every $\alpha\in (0,1)$ (equivalently, there exists $\alpha\in (0,1)$), there exist $T>0$ and $C>0$ such that, for every $y_0\in X$, there exists $u\in L^2(0,T;U)$ such that
\begin{equation}\label{alphaC}
\Vert y(T;y_0,u)\Vert_X\leq \alpha\Vert y_0\Vert_X\qquad\textrm{and}\qquad\Vert u\Vert_{L^2(0,T;U)}\leq C\Vert y_0\Vert_X.
\end{equation}
\item For every $\alpha\in(0,1)$ (equivalently, there exists $\alpha\in(0,1)$), there exist $T>0$ and $C>0$ such that
\begin{equation}\label{obs_stab}
\Vert S(T)^*\psi\Vert_{X} \leq C \Vert B^*S(T-\,\cdot\,)^*\psi\Vert_{L^2(0,T;U)} + \alpha\Vert\psi\Vert_X   \qquad\forall\psi\in X .
\end{equation}
\end{itemize}
When one of these items is satisfied, the smallest possible constant $C$ in \eqref{alphaC} and in the observability inequality \eqref{obs_stab} is $C=\mu_\alpha^T$.

Moreover, the best stabilization decay rate defined by \eqref{defomegastar} is
\begin{equation}\label{def_omegastar}
\omega^* = \inf \left\{ \frac{\ln\alpha}{T}\ \bigm|\ \alpha\in (0,1),\ T>0 \ \textrm{s.t.}\  \mu_\alpha^T<+\infty \right\}.
\end{equation}
\end{theorem}

Comments and remarks on these results are now in order.

\begin{remark}[On approximate controllability]
By the same approach, we obtain as well the following characterization of approximate controllability by an observability inequality:

\begin{quote}
{\it
Define  $\mu_{y_0,y_1,\alpha}^T\in[0,+\infty]$ similarly as $\mu_{y_0,\alpha}^T$, replacing the term $S(T)y_0$ with $S(T)y_0-y_1$. The control system \eqref{contsyst} is approximately controllable in time $T>0$ if and only, for all $y_0,y_1\in X$ and for every $\alpha>0$, there exists $C>0$ such that
$$
\langle\psi,S(T)y_0-y_1\rangle_{X} \leq C \Vert B^*S(T-\,\cdot\,)^*\psi\Vert_{L^2(0,T;U)} + \alpha\Vert y_0\Vert_X\Vert\psi\Vert_X  \qquad\forall\psi\in X .
$$
When $\mu_{y_0,y_1,\alpha}^T<+\infty$, it is the smallest constant $C$ in the observability inequality above.
}
\end{quote}
This statement underlines in an unusual way the difference between approximate controllability and approximate null controllability (recall that both notions coincide if $S(T)^*$ is injective, or equivalently if $\mathrm{Ran}(S(T))$ is dense in $X$, as it can be obviously recovered from the observability inequalities).

Similar statements can be obtained for $\alpha$-controllability to some target point $y_1$. We do not give details.
\end{remark}

\begin{remark}
With Theorem \ref{thm_approx}, we recover \cite[Proposition 1.17]{Boyer} where \eqref{obs_approx} is established by a contradiction argument and is written in this paper in the equivalent form of a sum of squares, i.e., with the present notations,
$$
\frac{1}{2}\langle\psi,S(T)y_0\rangle_{X}^2\leq  C \int_0^T \Vert B^*S(T-\,\cdot\,)^*\psi\Vert_U^2\, dt + \alpha^2\Vert y_0\Vert_X^2\Vert\psi\Vert_X^2 \qquad\forall\psi\in X .
$$
\end{remark}

\begin{remark}
From Theorems \ref{thm_alpha_null} and \ref{thm_approx}, we see that the constant $\mu_{y_0,\alpha}^T$ quantifies the $\alpha$-null controllability and null approximate controllability properties in some sense.
Actually, as established in the proof in Section \ref{sec_fenchel}, when $\mu_{y_0,\alpha}^T<+\infty$ we have
$$
\mu_{y_0,\alpha}^T=\Vert\bar u_{y_0,\alpha}\Vert_{L^2(0,T;U)}
$$
where $\bar u_{y_0,\alpha}$ is the (unique) control of minimal $L^2$ norm steering in time $T$ the control system \eqref{contsyst} from $y_0$ to the ball $\alpha\Vert y_0\Vert_X\overline B_1$ (and even, to the boundary of this ball, i.e., $\Vert y(T;y_0,\bar u_{y_0,\alpha})\Vert_X=\alpha\Vert y_0\Vert_X$).
\end{remark}

\begin{remark}
Using \eqref{alphaC}, we see that exponential stabilizability implies $\alpha$-null controllability in some time $T>0$ (depending on $\alpha$) for every $\alpha>0$.
\end{remark}

\begin{remark}\label{rem_approx_stab}
Neither approximate controllability nor exponential stabilizability imply each other:
\begin{itemize}
\item There exist control systems that are exponentially stabilizable but that are not approximately null controllable in any time $T$.

For instance, take $B=0$ and take $A$ generating an exponentially stable semigroup, i.e., $\Vert S(t)\Vert_{L(X)}\leq Me^{-\beta t}$ for some $M\geq 1$ and $\beta<0$, then a minimal time $T_\alpha$ is at least required to realize $\alpha$-null controllability is $T_\alpha=\frac{1}{\beta}\ln\frac{M}{\alpha}$, and $T_\alpha$ cannot be bounded uniformly with respect to every $\alpha>0$ arbitrarily small.
\item There exist control systems that are approximately controllable in some time $T$ but that are not exponentially stabilizable (see \cite[Example 5.2.2 page 228]{CurtainZwart} or \cite[Example 3.16]{PZ} or \cite[Theorem 3.3, (ii), page 227]{Z}).
\end{itemize}
\end{remark}

\begin{remark}\label{rem_muinfty}
By definition $\mu_{y_0,\alpha}^T$ depends on the initial point $y_0$.
We have seen that $\mu_{y_0,\alpha}^T$ is positively homogeneous with respect to $y_0$, i.e.,
$\mu_{y_0,\alpha}^T = \mu_{\frac{y_0}{\Vert y_0\Vert_X},\alpha}^T \Vert y_0\Vert_X$, which has led us to define $\mu_\alpha^T$ by \eqref{def_muT}.

It is then natural to wonder when $\mu_\alpha^T<+\infty$, i.e., when $\mu_{y_0,\alpha}^T$ is uniformly bounded on the unit sphere of $X$ (defined by $\Vert y_0\Vert_X=1$). When it happens, the constant $C$ in the observability inequalities \eqref{obs_alpha_null} and \eqref{obs_approx} is uniform with respect to all possible $y_0\in X$ such that $\Vert y_0\Vert_X=1$, and thus, taking $y_0=\frac{S(T)^*\psi}{\Vert S(T)^*\psi\Vert_X}$, we obtain the observability inequality \eqref{obs_stab} characterizing exponential stabilizability. Therefore:
\begin{itemize}
\item {\it Given some $\alpha\in(0,1)$, we have $\mu_\alpha^T<+\infty$ for some $T>0$ if and only if $\alpha$-null controllability in some time $T>0$ and exponential stabilizability are equivalent.}

\item We have the equivalence:
{\it there exists $T>0$ such that $\mu_\alpha^T<+\infty$ for every $\alpha>0$, if and only if approximate null controllability in some time $T>0$ and exponential stabilizability are equivalent.}

\end{itemize}
Note that there exist examples of control systems that are $\alpha$-null controllable or approximately null controllable in some time $T$ but not exponentially stabilizable (see Remark \ref{rem_approx_stab}).
\end{remark}

\begin{remark}\label{rem_exactnullcont}
If the control system \eqref{contsyst} is exactly null controllable in time $T$ then $\mu_\alpha^T<+\infty$ for every $\alpha>0$, and $\mu_\alpha^T$ has a limit as $\alpha\rightarrow 0^+$, denoted by $\mu_0^T$ (these facts are easily seen by considering the optimal controls $\bar u_{y_0,\alpha}$). In particular, we have the observability inequality
$$
\langle\psi,S(T)y_0\rangle_X \leq \mu_0^T \Vert B^*S(T-\,\cdot\,)^*\psi\Vert_{L^2(0,T;U)}
$$
for every $\psi\in X$ and every $y_0\in X$ of norm $1$.
Taking $y_0=\frac{S(T)^*\psi}{\Vert S(T)^*\psi\Vert_X}$, we recover the observability inequality \eqref{obs_null} corresponding to exact null controllability, and we have $C_T = \mu_0^T$.
Actually:
\begin{quote}
{\it
The control system \eqref{contsyst} is exactly null controllable in time $T$ if and only if $\mu_0^T=\lim_{\alpha\rightarrow 0^+}\mu_\alpha^T<+\infty$.
}
\end{quote}
\end{remark}

\begin{remark}
Recall that $\mu_\alpha^T=\mu_{y_0,\alpha}^T=0$ when $\alpha\geq\Vert S(T)\Vert_{L(X)}$.
When $\alpha\rightarrow 0^+$, in general $\mu_\alpha^T$ and $\mu_{y_0,\alpha}^T$ tend to $+\infty$. For instance, if the control system \eqref{contsyst} is approximately null controllable in time $T$ but not exactly null controllable in time $T$, then $\mu_{y_0,\alpha}^T\rightarrow+\infty$ as $\alpha\rightarrow 0^+$, for some $y_0\in X$.
It is then natural to wonder when those quantities remain uniformly bounded with respect $\alpha\in(0,\Vert S(T)\Vert_{L(X)}]$. We have the following fact:
\begin{quote}
{\it
Given $y_0\in X$ and $T>0$, if $\mu_{y_0,\alpha}^T$ is uniformly bounded with respect $\alpha\in(0,\Vert S(T)\Vert_{L(X)}]$, then  approximately null controllability in time $T$ is equivalent to exact null controllability in time $T$.
}
\end{quote}
Indeed, it follows from the arguments of proof in Section \ref{sec_fenchel} that, when $\mu_{y_0,\alpha}^T$ is uniformly bounded with respect $\alpha$, for every $y_0\in X$ there exists a (unique) optimal control $\bar u_{y_0,0}$ steering the control system \eqref{contsyst} from $y_0$ to $0$ in time $T$. Hence, in turn, $\mu_{y_0,\alpha}^T$ is bounded uniformly with respect to all $y_0\in X$ such that $\Vert y_0\Vert_X=1$.
\end{remark}

\begin{remark}\label{rem_exp_mu}
If the semigroup $(S(t))_{t\geq 0}$ is exponentially stable then the observability inequality \eqref{obs_stab} is obviously satisfied with $B=0$. Indeed, the semigroup $(S(t))_{t\geq 0}$ is exponentially stable if and only if there exists $T>0$ such that $\Vert S(T)\Vert_{L(X)}<1$. This is in accordance with the fact that, in this case, no control is required to stabilize the control system.

When the semigroup is not exponentially stable, the observability inequality \eqref{obs_stab} can be seen as a weakened version of the observability inequality \eqref{obs_null} corresponding to exact null controllability, by adding the term $\alpha\Vert\psi\Vert_X$ at the right-hand side for some $\alpha\in(0,1)$: this appears as a kind of compromise between the lack of exponential stability of $(S(t))_{t\geq 0}$ and the feedback action needed to exponentially stabilize the control system \eqref{contsyst}.
\end{remark}

\begin{remark}
Inspecting the observability inequalities \eqref{obs_null} and \eqref{obs_stab}, we recover the well known fact that if the control system \eqref{contsyst} is exactly null controllable in some time $T$ then it is exponentially stabilizable (see, e.g., \cite[Theorem 3.3 page 227]{Z}).
The converse is wrong in general: as said in Remark \ref{rem_exactnullcont}, the control system \eqref{contsyst} is exactly null controllable in time $T$ if and only if $\mu_0^T=\lim_{\alpha\rightarrow 0^+}\mu_\alpha^T<+\infty$; it may happen that when the system is exponentially stabilizable, as $\alpha\rightarrow 0^+$, the infimum of times $T$ such that $\mu_\alpha^T<+\infty$ tends to $+\infty$.

See Remark \ref{rem_complete_stab} hereafter for a stronger statement and for a partial converse.
\end{remark}

\begin{remark}[Complete stabilizability]\label{rem_complete_stab}
Complete stabilizability means that, given any $\omega\in\R$, one can find a feedback $K\in L(X,U)$ such that $\omega_K<\omega$, or equivalently, that the best stabilization rate given by \eqref{def_omegastar} is $-\infty$.
According to the expression \eqref{def_omegastar}, we have complete stabilizability if either, for a given $\alpha\in(0,1)$, we have $\mu_\alpha^T<+\infty$ for every $T>0$ arbitrarily small, or, for a given $T>0$, we have $\mu_\alpha^T<+\infty$ for every $\alpha>0$ arbitrarily small (which is equivalent, by Remark \ref{rem_exactnullcont}, to exact null controllability in time $T$).
Several remarks on complete stabilizability are in order.
\begin{itemize}
\item We have the following result:
\begin{quote}
{\it
If the control system \eqref{contsyst} is exactly null controllable in some time $T$ then it is completely stabilizable.
}
\end{quote}

This fact is proved in Proposition \ref{prop_exactnull_implies_complete_stab} in Appendix \ref{app:exactnull_implies_complete_stab}. We have not been able to find it in the existing literature (actually this fact is usually established under the additional assumption that $(S(t))_{t\geq 0}$ is a group).

This result also follows by Remark \ref{rem_exactnullcont} because, since $\mu_\alpha^T$ remains uniformly bounded as $\alpha\rightarrow 0$ for some $T>0$ fixed, it follows that $\omega^*=-\infty$.

\item Let us assume that $(S(t))_{t\geq 0}$ is a group. Then there exists $M\geq 1$ and $\omega\geq 0$ such that $\Vert S(t)\Vert_{L(X)}\leq Me^{\omega\vert t\vert}$ for every $t\in\R$.
If the control system \eqref{contsyst} is exponentially stabilizable, then by Theorem \ref{thm_stab}, using that $\Vert\psi\Vert_X\leq \Vert S(-T)\Vert_{L(X)}\Vert S(T)^*\psi\Vert_X$, we have
\begin{multline*}
\forall\alpha\in(0,1)\quad \exists T>0\quad \exists C>0\quad \textrm{s.t.}\\
\Vert S(T)^*\psi\Vert_X \leq C\Vert B^*S(T-\,\cdot\,)^*\psi\Vert_{L^2(0,T;U)} + \alpha Me^{\omega\vert T\vert}\Vert S(T)^*\psi\Vert_X\qquad\forall\psi\in X
\end{multline*}
and therefore we easily recover the following result:
\begin{quote}
{\it
When $(S(t))_{t\geq 0}$ is a group, we have equivalence of: exact controllability in some time $T$; exact null controllability in some time $T$; complete stabilizability.
}
\end{quote}
This result is already known (see \cite{Slemrod} and \cite[Theorem 3.4 page 229]{Z}).

The strategy developed in \cite{Komornik} (applying also, to some extent, to unbounded admissible control operators) consists of taking $K_\lambda = -B^* C_\lambda^{-1}$ where $C_\lambda$ is defined by
$$
C_\lambda = \int_0^{T+1/2\lambda} f_\lambda(t) S(-t) B B^* S(-t)^* \, dt
$$
(variant of the Gramian operator) with $\lambda>0$ arbitrary, $f_\lambda(t) = e^{-2\lambda t}$ if $t\in[0,T]$ and $f_\lambda(t) = 2\lambda e^{-2\lambda T}(T+1/2\lambda-t)$ if $t\in[T,T+1/2\lambda]$. The function
$V(y) = \langle y, C_\lambda^{-1} y \rangle $ is a Lyapunov function (as noticed in \cite{Coron}), and the feedback $K_\lambda$ yields exponential stability with rate $-\lambda$. We also refer to \cite{CL,PWX} for developments on rapid stabilization.

\item Let us assume that $A$ is skew-adjoint, which is equivalent, by the Stone theorem, to the fact that $A$ generates a unitary group $(S(t))_{t\in\R}$ (see \cite{EN}). Then $\Vert S(T)^*\psi\Vert_X=\Vert\psi\Vert_X$ for every $\psi\in X$ and therefore the observability inequality \eqref{obs_stab} characterizing exponential stabilizability is equivalent to the observability inequality \eqref{obs_exactcont} characterizing exact controllability and can be achieved, for a given $T>0$, with arbitrarily small values of $\alpha>0$, which implies that $\omega^*=-\infty$. Therefore we recover a result of \cite{Liu}:
\begin{quote}
{\it
When $A$ is skew-adjoint, we have equivalence of: exact controllability in time $T$; exact null controllability in time $T$; exponential stabilizability; complete stabilizability.
}
\end{quote}

\end{itemize}

\end{remark}

 \begin{remark}
Assume that the control system \eqref{contsyst} is  exponentially stabilizable.
By Theorem \ref{thm_stab}, for every $\alpha\in (0,1)$, there is $T>0$ such that
$\mu_\alpha^T<+\infty$. Defining $T_\alpha=\inf\{T>0 \ \mid\ \mu_\alpha^T<+\infty\}$, by \eqref{def_omegastar} we have
$$
\omega^* = \inf \left\{ \frac{\ln\alpha}{T_\alpha}\ \bigm| \ \alpha\in(0,1) \right\} .
$$
%
%
\end{remark}

%
%

\begin{remark}
Assume that $X=X_1\oplus X_2$ with $X_1$ and $X_2$ Banach subspaces of $X$ and that, in this decomposition, the operators $A$ and $B$ are written as
$$
A=\begin{pmatrix}
A_1 & 0\\
0 & A_2
\end{pmatrix} \quad\textrm{and}\quad
B = \begin{pmatrix}
B_1\\
B_2
\end{pmatrix}
$$
with $A_i:D(A_i)\subset X_i\rightarrow X_i$ generating a $C_0$-semigroup $(S_i(t))_{t\geq 0}$ on $X_i$, for $i=1,2$. 
Writing $y(t)=y_1(t)+y_2(t)\in X_1\oplus X_2$, the control system \eqref{contsyst} is equivalent to
\begin{align}
\dot y_1(t) &= A_1 y_1(t) + B_1 u(t) \label{subsys1} \\
\dot y_2(t) &= A_2 y_2(t) + B_2 u(t) \label{subsys2}
\end{align}
We have the following well known result, reminiscent of \cite{Triggiani_JMAA1975}:
\begin{quote}
{\it
If the control subsystem \eqref{subsys1} on $X_1$ is exponentially stabilizable and if the semigroup $(S_2(t))_{t\geq 0}$ is exponentially stable then the control system \eqref{contsyst} is exponentially stabilizable.
}
\end{quote}
Indeed, let $K_1\in L(X_1,U)$ be such that $A_1+B_1K_1$ generates an exponentially stable semigroup: this means that, given an initial condition $y(0)=y_1(0)+y_2(0)\in X_1\oplus X_2$, plugging the feedback control $u=K_1y_1$ into the control subsystem \eqref{subsys1} implies that $\Vert y_1(t)\Vert_{X_1}\leq M_1e^{\omega_1 t}$ for some $M_1>0$ and some $\omega_1<0$. Now, integrating \eqref{subsys2} (with the same control), we obtain as well that $y_2(t)$ converges exponentially to $0$ as $t\rightarrow +\infty$.

Note that, when $X_1$ is of finite dimension, the above feedback control is of finite dimensional nature. This procedure has been used and generalized in various contexts (see, e.g., \cite{CoronTrelat,Russell}).

The above result is in accordance with Theorem \ref{thm_stab}. Indeed, the control subsystem \eqref{subsys1} on $X_1$ is exponentially stabilizable and the semigroup $(S_2(t))_{t\geq 0}$ is exponentially stable, if and only if for every $\alpha_1\in(0,1)$ there exist $\alpha_2\in(0,1)$, $C_1>0$ and $T_1>0$ such that
\begin{align*}
\Vert S_1(T_1)^*\psi_1\Vert_{X_1} &\leq C_1 \Vert B_1^* S_1(T_1-\,\cdot\,)^*\psi_1\Vert_{L^2(0,T_1;U)} + \alpha_1\Vert\psi_1\Vert_{X_1}  & \forall\psi_1\in X_1 \\
\Vert S_2(T_1)^*\psi_2\Vert_{X_2} &\leq \alpha_2\Vert\psi_2\Vert_{X_2} & \forall\psi_2\in X_2
\end{align*}
It is easy to recover the observability inequality \eqref{obs_stab} from these two inequalities. We do not provide any details.
\end{remark}


\subsection{An example}
Let $\Omega\subset \mathbb{R}^d$ (with $d\geq 1$) be a smooth bounded domain. Let $\omega\subset \Omega$ be an open subset satisfying the Geometric Control Condition in some time $T$ (see \cite{BLR,LRLTT}).
We consider the controlled hyperbolic-parabolic coupled system:
\begin{equation}\label{e2.1}
\begin{cases}
z_{tt}(t,x)-\triangle z(t,x)= w(t,x)+\chi_\omega(x) u(t,x) \quad&(t,x)\in\Omega\times (0,+\infty),\\
w_{t}(t,x)-\triangle w(t,x)=0\quad&(t,x)\in\Omega\times (0,+\infty),\\
z(0,x)=w(0,x)=0 & x\in\partial\Omega,\\
\end{cases}
\end{equation}
with initial conditions $z(0,\cdot)\in H^1_0(\Omega)$, $z_t(0,\cdot)\in L^2(\Omega)$, $w(0,\cdot)\in L^2(\Omega)$.
This is a control system of the form \eqref{contsyst} with $X=H_0^1(\Omega)\times L^2(\Omega)\times L^2(\Omega)$, $U=L^2(\Omega)$ and
$$
A=\begin{pmatrix}
0 &1 &0\\
\triangle&0&1\\
0&0&\triangle
\end{pmatrix}
\quad\textrm{and}\quad
B=\begin{pmatrix}
0 \\
\chi_\omega\\
0
\end{pmatrix}.
$$
We claim that:
\begin{itemize}
\item The control system \eqref{e2.1} is not exactly null controllable, in any time $T$.

Indeed, the second equation in  \eqref{e2.1} is not controlled.

\item The control system \eqref{e2.1} with the zero control $u=0$ is not asymptotically stable.

Indeed, taking $w(0,\cdot)=0$, we have $w(t,\cdot)=0$ for every $t\geq 0$ and then, since $u=0$, we have conservation of the $H^1\times L^2$ norm of $z$ thanks to the first equation in \eqref{e2.1}.

\item The control system \eqref{e2.1} is exponentially stabilizable.
\end{itemize}
The last item is proved thanks to Theorem \ref{thm_stab}. The argument goes as follows.

The adjoint system is
\begin{equation}\label{e2.2}
\begin{cases}
\phi_{tt}(t,x)-\triangle\phi(t,x) =0,\quad&(t,x)\in\Omega\times (0,T),\\
\xi_{t}(t,x)+\triangle\xi(t,x)+\phi_t(t,x)=0\quad& (t,x)\in\Omega\times (0,T),\\
\xi(0,x)=\phi(0,x)=0,& x\in\partial\Omega,\\
\end{cases}
\end{equation}
with $\phi(\cdot,T)\in L^2(\Omega)$, $\phi_t(\cdot,T)\in H^{-1}(\Omega)$, $\xi(T,\cdot)\in L^2(\Omega)$.
Since $(\omega,T)$ satisfies the Geometric Control Condition, the adjoint system is observable in time $T$ (see \cite{BLR,LRLTT}) and thus there exists $C>0$ such that
\begin{equation}\label{e2.3}
\Vert (\phi(0), \phi_t(0)\Vert_{L^2(\Omega)\times H^{-1}(\Omega)}^2 \leq C\int^T_0\int_\omega\phi(t,x)^2\, dt\, dx
\end{equation}
for any solution of \eqref{e2.2}.
Let $\lambda_1$ be the first eigenvalue of the Dirichlet Laplacian $-\triangle$. Let  $T>\sqrt{2}/2$ satisfying that
\begin{equation}\label{e2.4}
e^{-2\lambda_1T}\leq \frac{1}{8}.
\end{equation}
We infer from the second equation of \eqref{e2.2} that
$$
\|\xi(0)\|^2=\left\|e^{\triangle T }\xi(T)-\int^T_0e^{\triangle s}\phi_t(s)\, ds\right\|^2
\leq 2\left\|e^{\triangle T}\xi(T)\right\|^2+2\left\|\int^T_0e^{\triangle s}\phi_t(s)\, ds\right\|^2
$$
and then, using that $\left\|e^{\triangle t }\right\|\leq e^{-\lambda_1t}\leq 1$ for every $t\geq 0$, we obtain
\begin{multline*}
\|\xi(0)\|^2
\leq 2e^{-2\lambda_1T}\|\xi(T)\|^2+2\left\|\int^T_0\|\phi_t(s)\|\, ds\right\|^2
\leq 2e^{-2\lambda_1T}\|\xi(T)\|^2+2T\int^T_0\|\phi_t(s)\|^2\, ds\\
\leq 2e^{-2\lambda_1T}\|\xi(T)\|^2+2T\int^T_0\left\| (\phi(s), \phi_t(s))\right\|^2\, ds
\end{multline*}
and therefore, using the conservation of energy for the wave equation and \eqref{e2.4},
\begin{equation}\label{e2.5}
\|\xi(0)\|^2\leq2e^{-2\lambda_1T}\|\xi(T)\|^2+2T^2\left\| \phi(0), \phi_t(0) \right\|^2
\leq \frac{1}{4}\|\xi(T)\|^2+2T^2\left\| (\phi(0), \phi_t(0))\right\|^2.
\end{equation}
Multiplying \eqref{e2.3} by $4T^2$ and using \eqref{e2.5}, we obtain that
\begin{multline*}
\left\| (\phi(0), \phi_t(0))\right\|^2+\|\xi(0)\|^2\leq 2T^2 \left\| (\phi(0), \phi_t(0)) \right\|^2+\|\xi(0)\|^2\\
\leq\frac{1}{4}\|\xi(T)\|^2
+4T^2C\int^T_0\int_\omega\phi(t,x)^2\, dt\, dx
\end{multline*}
This leads to
$$
2\left\| (\phi(0), \phi_t(0), \xi(0))\right\|
\leq \|\xi(T)\|
+4T\sqrt{C}\|\chi_\omega\phi\|
\leq\left\| (\phi(T), \phi_t(T), \xi(T)) \right\|+4T\sqrt{C}\|\chi_\omega\phi\|.
$$
Using Theorem \ref{thm_stab}, we conclude that the control system \eqref{e2.1} is exponentially stabilizable.

\subsection{Extensions, further comments and open problems}\label{sec:furthercomments}
\paragraph{Banach spaces.}
Throughout the paper we have assumed that the state space $X$ and the control space $U$ are Hilbert spaces. All our results can be extended without difficulty to the case where $X$ and $U$ are reflexive Banach spaces. One has to be careful to replace, in all observability inequalities, the scalar product $\langle\cdot,\cdot\rangle_X$ with the duality bracket $\langle\cdot,\cdot\rangle_{X',X}$.

\paragraph{Unbounded control operators.} We have assumed  that the control operator $B$ is linear continuous with values in $X$, i.e., $B\in L(U,X)$. This assumption covers the case of internal controls, but not, for instance, of boundary controls. For the latter case, we speak of \emph{unbounded control operators}, which are operators $B$ that are not continuous from $U$ to $X$ but are continuous from $U$ to some larger space $X_{-\alpha}$ which can be defined by extrapolation (scale of Banach spaces, see \cite{EN,Staffans,TW}). Given such a control operator $B\in L(U,X_{-\alpha})$ with $\alpha>0$, the range of the operator $L_T$ may then fail to be contained in $X$. We say that $B$ is \emph{admissible} when $\mathrm{Ran}(L_T)\subset X$ for some (and thus for all) $T>0$.
Note that, assuming that $X$ and $U$ are reflexive, it is known that if $B$ is admissible then $B\in L(U,X_{-1/2})$ (see \cite{TW}).

For an admissible control operator, since $\mathrm{Ran}(L_T)\subset X$, all arguments of Section \ref{sec_fenchel} remain valid.
One should anyway avoid to use $G_T^{1/2}$ in this more general context and replace $\Vert G_T^{1/2}\psi\Vert_X$ with $\Vert B^*S(T-\,\cdot\,)^*\psi\Vert_{L^2(0,T;U)}$ everywhere throughout the proof.
Therefore:

\begin{theorem}
Theorems \ref{thm_alpha_null} and \ref{thm_approx} are true in the more general context where the control operator $B$ may be unbounded but is admissible.
\end{theorem}

Now, concerning the exponential stabilizability result, the critical fact is in the proof of Lemma \ref{lem_expstab} where we use Riccati theory: indeed this theory is well established in the general case only for admissible control operators and analytic semigroups (see Remark \ref{remRiccati} at the end of the proof in Section \ref{sec_expstab}). Therefore:

\begin{theorem}\label{thm_stab_extended}
Theorem \ref{thm_stab} is true in the more general context where the control operator $B$ is unbounded and admissible and  the semigroup $(S(t))_{t\geq 0}$ is analytic.
\end{theorem}

For instance, this situation covers the case of heat equations in a $C^2$ bounded open subset $\Omega\subset\R^n$, with $X=L^2(\Omega)$, with Neumann control at the boundary of $\Omega$ (for which, at best, $B\in L(U,X_{-1/4-\varepsilon})$ for every $\varepsilon>0$), but not with Dirichlet control at the boundary (for which, at best, $B\in L(U,X_{-3/4-\varepsilon})$ for every $\varepsilon>0$). We refer to \cite{LT1} for details.

Since the critical fact is only in the use of Riccati theory in Lemma \ref{lem_expstab}, there are anyway some particular situations where the semigroup is not analytic, the control operator is admissible, and however Riccati theory is established (see \cite{LT1}). But in full generality the question is open.

\paragraph{Systems with an observation}
Throughout the paper we have focused on the control system \eqref{contsyst},
without observation.
We could add to the system an observation $z(t)=Cy(t)$, where $C\in L(X,Y)$ is an observation operator, with $Y$ another Hilbert space. Corresponding notions of controllability and of stabilizability are classically defined as well.
Extending our main results to that context is an interesting issue.

\paragraph{Discretization problems.}
The observability inequality \eqref{obs_stab} may certainly be exploited to recover results on uniform semi-discretizations (or full discretizations). The problem is the following. Consider a spatial semi-discrete model of \eqref{contsyst}, written as
\begin{equation}\label{contsystN}
\dot y_N=A_Ny_N+B_Nu_N
\end{equation}
with $y_N(t)\in\R^N$ (see \cite{Boyer,LabbeTrelat,LT1,LT2} for a general framework on discretization issues). Assuming that the control system \eqref{contsyst} is exponentially stabilizable (equivalently, that the observability inequality \eqref{obs_stab} is satisfied), how to ensure that the family of control systems \eqref{contsystN} is \emph{uniformly} exponentially stabilizable? Uniform means here that, for each $N$, there exists a feedback matrix $K_N$ such that $\Vert\exp(t(A_B+B_NK_N))\Vert\leq Me^{\omega t}$ for some $M\geq 1$ and some $\omega<0$ that are \emph{uniform} with respect to $N$.

This problem has been much studied in the literature with various approaches.
In \cite{BanksIto1997, LT1, LiuZheng}, the convergence of the Riccati matrix $P_N$ (corresponding to \eqref{contsystN}) to the Riccati operator $P$ (corresponding to \eqref{contsyst}) is proved in the general parabolic case, even for unbounded control operators, that is, when $A:D(A)\rightarrow X$ generates an analytic semigroup, $B\in L(U,D(A^*)')$, and $(A,B)$ is exponentially stabilizable.
Uniform exponential stability is also proved under uniform Huang-Pr\"uss conditions in \cite{LiuZheng}, allowing to obtain convergence of Riccati operators for second-order systems $\ddot y + A y = Bu$ with $A:D(A)\rightarrow X$ positive selfadjoint with compact inverse and $B$ bounded control operator. General conservative equations are treated in \cite{EZ1} where it is proved that adding a viscosity term in the numerical scheme helps to recover a uniform exponential decay, provided a uniform observability inequality holds true for the corresponding conservative equation (see also \cite{Alabau,Trelat_stab2018} for semilinear equations).
Of course, given a specific equation, the difficulty is to establish the uniform observability inequality. This is a difficult issue, investigated in some particular cases (see \cite{ZuazuaSIREV} for a survey).

Anyway, the observability inequality \eqref{obs_stab} is written in the semidiscrete case as
$$
\Vert S_N(T)^*\psi_N\Vert \leq C \left( \int_0^T \Vert B_N^*S(T-t)^*\psi_N\Vert^2\, dt\right)^{1/2} + \alpha\Vert\psi_N\Vert
$$
with $\alpha\in(0,1)$ and $C>0$ uniform with respect to $N$. Such a uniform inequality is likely to be true in many cases. For instance it is nicely shown in \cite[Section 5]{Zuazua_RendiConti} that approximate boundary controllability for 1D waves withstands (in a uniform way) a finite-difference semi-discretization scheme. This is shown by using a functional similar to the one \eqref{defJ} used here (see also \cite{Boyer,LabbeTrelat} for uniform exact null controllability in the parabolic case).

But such considerations go beyond the scope of the present paper. We think that the observability inequalities derived in our main results may be used, at least to recover some known results on uniform convergence, and maybe to establish new ones.

\paragraph{Hautus test.} There exist many results with variants of Hautus tests. For instance, when $(S(t))_{t\geq 0}$ is a normal $C_0$-group, the Hautus test property
\begin{equation*}
\exists C>0\quad \mid\quad \Vert (\lambda\mathrm{id}-A^*)\psi\Vert_X^2+\vert\mathrm{Re}(\lambda)\Vert B^*\psi\Vert_X^2\geq C\vert\mathrm{Re}(\lambda)\vert^2\Vert\psi\Vert_X^2\qquad\forall \lambda\in C_-\quad\forall\psi\in D(A^*)
\end{equation*}
where $\C_-$ is the open left complex half-plane, is sufficient to ensure exponential stabilizability (see \cite{JacobZwart_2009}).
The question of how Hautus tests are related to the observability inequalities derived in our paper is open.

\paragraph{Polynomial stabilizability.}
We have provided in Theorem \ref{thm_stab} a characterization of exponential stabilizability.
The $C_0$-semigroup $(S(t))_{t\geq 0}$ is said to be polynomially stable when there exist constants $\gamma,\delta>0$ such that $\Vert S(t)(A-\beta\mathrm{id})^{-\gamma}\Vert_{L(X)}\leq M t^{-\delta}$ for every $t\geq 1$, for some $M>0$ and some $\beta\in\rho(A)$ (see \cite{JacobSchnaubelt} where polynomial stability is compared with observability).
Finding a dual characterization of polynomial stabilizability in terms of an observability inequality is an open issue, which may be related to the previous question on Hautus tests.

%
%

\section{Proofs}\label{sec:proofs}
In this section, we provide  proofs of Theorems \ref{thm_alpha_null}, \ref{thm_approx} and \ref{thm_stab}, following Fenchel duality arguments. In turn, we establish intermediate results which may have interest in themselves for other purposes.
%
%

\subsection{Fenchel duality arguments}\label{sec_fenchel}
We start our analysis by considering the $\alpha$-null controllability problem or the approximate null controllability problem for the control system \eqref{contsyst}.
What is written hereafter in this first subsection essentially follows the classical analysis by Fenchel duality done in \cite{Lions_1992} (see also \cite{GlowinskiLionsHe}), but is written in a more general framework.

Let $\alpha\geq 0$ and let $T>0$ be arbitrary. Let $\overline B_1$ be the closed unit ball in $X$. Given an arbitrary $y_0\in X$, we consider the problem of steering the control system \eqref{contsyst} from $y_0$ to $\alpha\Vert y_0\Vert_X\overline B_1$ in time $T$, i.e.,
$$
\Vert y(T;y_0,u)\Vert_X \leq \alpha\Vert y_0\Vert_X
$$
($\alpha\Vert y_0\Vert_X$-null controllability problem).
When it is possible, i.e., when the control system \eqref{contsyst} is $\alpha\Vert y_0\Vert_X$-null controllable, we search the control of minimal $L^2$ norm (which is unique by strict convexity). We set
$$
S_{y_0,\alpha}^T = \inf \left\{ \frac{1}{2} \Vert u\Vert_{L^2(0,T;U)}^2 \quad \bigm| \quad y(T;y_0,u)\in \alpha\Vert y_0\Vert_X\overline B_1 \right\}
$$
with the convention that $S_{y_0,\alpha}^T = +\infty$ whenever $\alpha\Vert y_0\Vert_X\overline B_1$ is not reachable from $y_0$ in time $T$.

The following lemma is obvious.

\begin{lemma}\label{lemS}
The control system \eqref{contsyst} is:
\begin{itemize}
\item $\alpha$-null controllable in time $T$ if and only if $S_{y_0,\alpha}^T<+\infty$ for every $y_0\in X$ (where $\alpha\geq 0$ is arbitrary);
\item approximately null controllable in time $T$ if and only if $S_{y_0,\alpha}^T<+\infty$ for every $y_0\in X$ and for every $\alpha>0$;
\end{itemize}
\end{lemma}

\subsubsection{Application of Fenchel duality}
Following \cite{Lions_1992}, we define the convex and lower semi-continuous functions $F:L^2(0,T;U)\rightarrow[0,+\infty)$ and $G:X\rightarrow[0,+\infty]$ by
$$
F(u) =\frac{1}{2} \Vert u\Vert_{L^2(0,T;U)}^2
$$
and
\begin{equation*}
G(\varphi) = \left\{ \begin{array}{ll}
0 & \textrm{if}\ \varphi\in -S(T)y_0+\alpha\Vert y_0\Vert_X\overline B_1 \\
+\infty & \textrm{otherwise}
\end{array}\right.
\end{equation*}
and we note that
$$
S_{y_0,\alpha}^T = \inf_{u\in L^2(0,T;U)} ( F(u)+G(L_Tu) ) .
$$
The Fenchel conjugates $F^*:L^2(0,T;U)\rightarrow[0,+\infty)$ and $G^*:X\rightarrow[0,+\infty]$ are given by
$$
F^*(v) = \sup_{u\in L^2(0,T;U)} \left( \langle v,u\rangle_X-F(u) \right) = \frac{1}{2} \Vert v\Vert^2_{L^2(0,T;U)} = F(v)
$$
and
$$
G^*(\psi) = \sup_{\varphi\in X} \left( \langle \varphi,\psi\rangle_X-G(\varphi)\right)
= \sup_{\varphi\in -S(T)y_0+\alpha\Vert y_0\Vert_X\overline B_1} \langle\varphi,\psi\rangle_X
= - \langle\psi,S(T)y_0\rangle_X+\alpha\Vert y_0\Vert_X\Vert\psi\Vert_X
$$
where the latter equality is obtained by applying the Cauchy-Schwarz inequality.

Noting that $L_T\,\mathrm{dom(F)} = \mathrm{Ran}(L_T)$ intersects the set of points at which $G$ is continuous when \eqref{contsyst} is either $\alpha$-null controllable in time $T$ (with $\alpha>0$ fixed) or approximately null controllable in time $T$, we infer from the Fenchel-Rockafellar duality theorem (see \cite{Fenchel,Rockafellar}) that
$$
S_{y_0,\alpha}^T = - \inf_{\psi\in X} \left( F^*(L_T^*\psi)+G^*(\psi) \right)
= - \inf_{\psi\in X} J_{y_0,\alpha}^T(\psi)
$$
where we have set
\begin{equation}\label{defJ}
J_{y_0,\alpha}^T(\psi) = F^*(L_T^*\psi)+G^*(\psi) = \frac{1}{2}\langle G_T\psi,\psi\rangle_X - \langle\psi,S(T)y_0\rangle_X + \alpha\Vert y_0\Vert_X\Vert\psi\Vert_X
\end{equation}
for every $\psi\in X$.
Note that $J_{y_0,\alpha}^T$ is differentiable except at $\psi=0$.

This result is still valid for $\alpha=0$ but is not a consequence of the Fenchel-Rockafellar duality theorem (because we have used in a critical way that $\alpha>0$): for $\alpha=0$ this is the usual procedure in the Hilbert Uniqueness Method (see \cite{GlowinskiLionsHe,Lions_HUM}).

\subsubsection{Computation of the minimizer}
As said above, when $S_{y_0,\alpha}^T<+\infty$, there is a unique minimizer $\bar u_{y_0,\alpha}\in L^2(0,T;U)$ and there is also a unique minimizer $\bar\psi_{y_0,\alpha}\in X$ of $J_{y_0,\alpha}^T$ (this follows from \eqref{eqgrad} below), and $S_{y_0,\alpha}^T = \frac{1}{2}\Vert\bar u_{y_0,\alpha}\Vert_{L^2(0,T;U)}^2 = -J_{y_0,\alpha}^T(\bar\psi_{y_0,\alpha})$.
We have either $\bar\psi_{y_0,\alpha}=0$ and then $\bar u_{y_0,\alpha}=0$ and $S_{y_0,\alpha}^T = 0$, or $\bar\psi_{y_0,\alpha}\neq 0$ and then $\nabla J_{y_0,\alpha}^T(\bar\psi_{y_0,\alpha})=0$, which gives
\begin{equation}\label{eqgrad}
G_T\bar\psi_{y_0,\alpha}-S(T)y_0+\alpha\Vert y_0\Vert_X\frac{\bar\psi_{y_0,\alpha}}{\Vert\bar\psi_{y_0,\alpha}\Vert_X} = 0 .
\end{equation}
Given any $\psi\in X$, we set $\psi=r\sigma$ with $r=\Vert\psi\Vert_X$ and $\sigma\in X$ of norm $1$ (polar coordinates). For the minimizer $\bar\psi_{y_0,\alpha}=\bar r_{y_0,\alpha}\bar\sigma_{y_0,\alpha}$, we infer from \eqref{eqgrad} that
$$
\bar r_{y_0,\alpha} = \frac{\langle S(T)y_0,\bar\sigma_{y_0,\alpha}\rangle_X-\alpha\Vert y_0\Vert_X}{\langle G_T\bar\sigma_{y_0,\alpha},\bar\sigma_{y_0,\alpha}\rangle_X} ,\qquad
\bar\sigma_{y_0,\alpha} = (\bar r_{y_0,\alpha} G_T+\alpha\Vert y_0\Vert_X)^{-1} S(T)y_0 .
$$
Here, we used the facts that $\langle G_T\bar\sigma_{y_0,\alpha},\bar\sigma_{y_0,\alpha}\rangle_X\neq0$ (which follows from \eqref{eqgrad}) and that $\bar u_{y_0,\alpha}\neq 0$.
Note that, necessarily, $\langle S(T)y_0,\bar\sigma_{y_0,\alpha}\rangle_X-\alpha\Vert y_0\Vert_X\geq 0$.

Note also that, in the Fenchel duality argument, the optimal control $\bar u_{y_0,\alpha}$ is given in function of $\bar\psi_{y_0,\alpha}$ by
$$
\bar u_{y_0,\alpha}(t) = (L_T^* \bar\psi_{y_0,\alpha})(t) = B^*S(T-t)^*\bar\psi_{y_0,\alpha} .
$$

Until that step, there is nothing new with respect to the existing literature. Up to our knowledge, the novelty is in the next step, with a simple remark leading to an observability inequality.

\subsubsection{An alternative optimization problem}
Following the above arguments, we first note that $J_{y_0,\alpha}^T(\psi) = J_{y_0,\alpha}^T(r\sigma) = \frac{1}{2}r^2\langle G_T\sigma,\sigma\rangle_X-r\left( \langle\sigma,S(T)y_0\rangle_X-\alpha\Vert y_0\Vert_X\right)$ and that $J_{y_0,\alpha}^T(0)=0$, and hence\footnote{Given $a\geq 0$ and $b\in\R$, we have $\displaystyle\inf_{r>0}(ar^2-br)=\left\{\begin{array}{ll}0 & \textrm{if}\ b\leq 0\\ -\frac{b^2}{4a} \ (-\infty\ \textrm{if}\ a=0) & \textrm{if}\ b>0,\ \textrm{reached at}\ r=\frac{b}{2a}\end{array}\right.$}, for any fixed $\sigma\in X$,
$$
\inf_{r>0} J_{y_0,\alpha}^T(r\sigma) =
\left\{\begin{array}{ll}
0 & \textrm{if}\ \langle\sigma,S(T)y_0\rangle_X-\alpha\Vert y_0\Vert_X \leq 0, \\
\displaystyle  - \frac{1}{2} \frac{ \left( \langle\sigma,S(T)y_0\rangle_X-\alpha\Vert y_0\Vert_X \right)^2 }{\langle G_T\sigma,\sigma\rangle_X} & \textrm{if}\ G_T\sigma\neq 0\ \textrm{and}\ \langle\sigma,S(T)y_0\rangle_X-\alpha\Vert y_0\Vert_X >0, \\
-\infty & \textrm{if}\ G_T\sigma=0\ \textrm{and}\ \langle\sigma,S(T)y_0\rangle_X-\alpha\Vert y_0\Vert_X >0.
\end{array}\right.
$$
Besides, by the definition \eqref{def_muy0alpha} of $\mu_{y_0,\alpha}^T$, we have
$$
\mu_{y_0,\alpha}^T =  \sup_{\psi\in X}  \left\{\begin{array}{ll}
0 & \textrm{if}\ \langle\psi,S(T)y_0\rangle_X-\alpha\Vert y_0\Vert_X\Vert\psi\Vert_X \leq 0, \\
\displaystyle \frac{ \langle\psi,S(T)y_0\rangle_X-\alpha\Vert y_0\Vert_X\Vert\psi\Vert_X }{\Vert G_T^{1/2}\psi\Vert_X} & \textrm{if}\ G_T\psi\neq 0\ \textrm{and}\ \langle\psi,S(T)y_0\rangle_X-\alpha\Vert y_0\Vert_X\Vert\psi\Vert_X >0, \\
+\infty & \textrm{if}\ G_T\psi=0\ \textrm{and}\ \langle\psi,S(T)y_0\rangle_X-\alpha\Vert y_0\Vert_X\Vert\psi\Vert_X >0.
\end{array}\right.
$$
It follows that
$$
S_{y_0,\alpha}^T = -\inf_{\psi\in X} J_{y_0,\alpha}^T(\psi) = -\inf_{\Vert\sigma\Vert_X=1} \inf_{r>0} J_{y_0,\alpha}^T(r\sigma) = \frac{1}{2} \left( \mu_{y_0,\alpha}^T \right)^2.
$$
Note that, when $0<\mu_{y_0,\alpha}^T < +\infty$, we have
$$
\mu_{y_0,\alpha}^T = \Vert\bar u_{y_0,\alpha}\Vert_{L^2(0,T;U)} = \frac{\langle\bar\psi_{y_0,\alpha},S(T)y_0\rangle_X-\alpha\Vert y_0\Vert_X\Vert\bar\psi_{y_0,\alpha}\Vert_X}{\Vert G_T^{1/2}\bar\psi_{y_0,\alpha}\Vert_X} .
$$

\begin{remark}\label{rem_muy0T}
According to  the above discussions, when   $\mu_{y_0,\alpha}^T<+\infty$, it is the smallest  constant $C\in[0,+\infty)$ such that there exists $u\in L^2(0,T;U)$ such that $\Vert y(T;y_0,u)\Vert_X\leq \alpha\Vert y_0\Vert_X$ and $\Vert u\Vert_{L^2(0,T;U)}\leq C$.
When $C=\mu_{y_0,\alpha}^T$, one has $u=\bar u_{y_0,\alpha}$.
\end{remark}


\subsubsection{Conclusion}
Theorems \ref{thm_alpha_null} and \ref{thm_approx} follow, using Lemma \ref{lemS} and the fact that $S_{y_0,\alpha}^T = \frac{1}{2} \left( \mu_{y_0,\alpha}^T \right)^2$.

\subsection{Fenchel dualization of the exponential stabilization property}\label{sec_expstab}
We start by proving \eqref{connectionconsts10}.

\begin{lemma}\label{appendixprop2}
Given $\alpha>0$ and $T>0$, $\mu_\alpha^T$ is the smallest possible  $C\in[0,+\infty]$ such that the following weak observability inequality is satisfied:
$$
\Vert S(T)^*\psi\Vert_X-\alpha\Vert\psi\Vert_X \leq C \Vert G_T^{1/2}\psi\Vert_X\qquad\forall\psi\in X
$$
and this, independently on the sign of $\Vert S(T)^*\psi\Vert_X-\alpha\Vert\psi\Vert_X$.
Moreover, when $\mu_\alpha^T<+\infty$, it is the smallest constant $C\in[0,+\infty)$
such that the above weak observability inequality holds.
\end{lemma}

\begin{proof}
Using \eqref{def_muy0alpha}, we have
\begin{equation}\label{obser-const-connection}
\mu_{y_0,\alpha}^T= \sup\limits_{\psi\in X}F_\alpha^T(y_0,\psi)
\end{equation}
where
$$
F_\alpha^T(y_0,\psi)=
\begin{cases}
\max\left( \displaystyle\frac{\langle\psi,S(T)y_0\rangle_X-\alpha\Vert y_0\Vert_X\Vert\psi\Vert_X}{\Vert G_T^{1/2}\psi\Vert_X},\,0\right)
 \quad &\text{if}\; (y_0,\psi)\in D_1,\\
 +\infty &\text{if}\;(y_0,\psi)\in D_2,\\
 ~0&\text{if}\;(y_0,\psi)\in D_3,
\end{cases}
$$
with
\begin{align*}
D_1&=\left\{(y_0,\psi)\in X\times X\ \mid\ G_T\psi\neq0\right\},\\
D_2&=\left\{(y_0,\psi)\in X\times X\ \mid\  G_T\psi=0\text{ and }\langle\psi,S(T)y_0\rangle_X-\alpha\Vert y_0\Vert_X\Vert\psi\Vert_X>0 \right\},\\
D_3&=\left\{(y_0,\psi)\in X\times X\ \mid\ G_T\psi=0\text{ and }\langle\psi,S(T)y_0\rangle_X-\alpha\Vert y_0\Vert_X\Vert\psi\Vert_X\leq0\right\}.
\end{align*}
Since
$$
\sup_{\Vert y_0\Vert_X=1}\sup\limits_{\psi\in X}F_\alpha^T(y_0,\psi)
 =\sup\limits_{\psi\in X}\sup_{\Vert y_0\Vert_X=1}F_\alpha^T(y_0,\psi)
$$
and
$$
\sup\limits_{\|y_0\|_X=1}F_\alpha^T(y_0,\psi)=
\begin{cases}
\max\left( \displaystyle\frac{\Vert S(T)^*\psi\Vert_X-\alpha\Vert\psi\Vert_X}{\Vert G_T^{1/2}\psi\Vert_X},\,0\right)
 \quad &\text{if}\; G_T\psi\neq0,\\
 +\infty &\text{if}\; G_T\psi=0,\,\exists\, y_0\text{ s.t.}(y_0,\psi)\in D_2,\\
 ~0&\text{if}\; G_T\psi=0,\,\nexists\, y_0\text{ s.t.}(y_0,\psi)\in D_2,
\end{cases}
$$
we derive the desired result from \eqref{def_muT} and \eqref{obser-const-connection}.
\end{proof}

We now give another interpretation of $\mu_\alpha^T$, useful to address exponential stabilizability.

\begin{lemma}\label{lem_interpret_muT}
Let $\alpha>0$ and $T>0$ be such that $\mu_\alpha^T<+\infty$. Then
$\mu_\alpha^T$ is the smallest  constant $C\geq 0$ such that, for every $y_0\in X$, there exists $u\in L^2(0,T;U)$ such that $\Vert y(T;y_0,u)\Vert_X\leq \alpha\Vert y_0\Vert_X$ and $\Vert u\Vert_{L^2(0,T;U)}\leq C\Vert y_0\Vert_X$.
\end{lemma}

\begin{proof}
The result follows from the facts that $S_{y_0,\alpha}^T = \frac{1}{2} \left( \mu_{y_0,\alpha}^T \right)^2 = \frac{1}{2}\Vert \bar u_{y_0,\alpha} \Vert_{L^2(0,T;U)}^2$, that $\mu_{y_0,\alpha}^T = \mu_{\frac{y_0}{\Vert y_0\Vert_X},\alpha}^T \Vert y_0\Vert_X\leq \mu_\alpha^T\Vert y_0\Vert_X$ and from Remark \ref{rem_muy0T}.
\end{proof}

Lemma \ref{lem_interpret_muT} is closely related to exponential stabilizability when $\alpha<1$. We indeed have the following result (easy consequence of well known results, however we provide a proof).

\begin{lemma}\label{lem_expstab}
The control system \eqref{contsyst} is exponentially stabilizable if and only if, for every $\alpha\in (0,1)$ (equivalently, there exists $\alpha\in (0,1)$), there exist $T>0$ and $C>0$ such that, for every $y_0\in X$, there exists $u\in L^2(0,T;U)$ such that $\Vert y(T;y_0,u)\Vert_X\leq \alpha\Vert y_0\Vert_X$ and $\Vert u\Vert_{L^2(0,T;U)}\leq C\Vert y_0\Vert_X$.

When this is satisfied, the best stabilization decay rate $\omega^*$ is the infimum of $\frac{\ln\alpha}{T}$ over all possible couples $(T,\alpha)$ for which the above inequalities are satisfied for some constant $C>0$.
\end{lemma}

\begin{proof}
Assume that the control system \eqref{contsyst} is exponentially stabilizable. Then there exists $K\in L(X,U)$ such that $\Vert S_K(t)\Vert_{L(X)}\leq Me^{\omega_Kt}$ for every $t\geq 0$, with $M\geq 1$ and $\omega_K<0$. Let $y_0\in X$. We set $u(t)=KS_K(t)y_0$ where $y(t)=S_K(t)y_0$ is the solution to $\dot y(t)=Ay(t)+BKy(t)$, $y(0)=y_0$. We have $\Vert y(t;y_0,u)\Vert_X=\Vert S_K(t)y_0\Vert_X\leq Me^{\omega_Kt}\Vert y_0\Vert_X=\alpha\Vert y_0\Vert_X$ for $T=\frac{1}{\omega_K}\ln\frac{\alpha}{M}$ and we compute
$\Vert u\Vert_{L^2(0,T;U)}\leq\frac{\Vert K\Vert_{L(X,U)}M}{\sqrt{-2\omega_K}}\|y_0\|_X$, whence the result.

Conversely, we proceed by iteration. For the initial condition $y_0$, there exists a control $u_0\in L^2(0,T;U)$ such that $\Vert y(T;y_0,u_0)\Vert_X\leq\alpha\Vert y_0\Vert_X$ and $\Vert u_0\Vert_{L^2(0,T;U)}\leq C\Vert y_0\Vert_X$. We set $y_1=y(T;y_0,u_0)$ and we repeat the argument for this new initial condition $y_1$, and then we iterate, obtaining that
$\Vert y_{j+1}\Vert_X=\Vert y((j+1)T;y_j,u_j)\Vert_X\leq\alpha\Vert y_j\Vert_X$ and $\Vert u_j\Vert_{L^2(0,T;U)}\leq C\Vert y_j\Vert_X$. The control $u$ defined as the concatenation of the controls $u_j\in L^2(jT,(j+1)T;U)$ generates the trajectory $y(\cdot;y_0,u)$ satisfying, at time $kT$, $\Vert y(kT;y_0,u)\Vert_X\leq \alpha^k\Vert y_0\Vert_X$, and
$$
\Vert u\Vert_{L^2(0,+\infty;U)}^2 \leq \sum_{j=0}^{+\infty} C^2 \Vert y(jT;y_0,u)\Vert_X^2 \leq C^2 \sum_{j=0}^{+\infty} \alpha^{2j}\Vert y_0\Vert_X^2 = \frac{C^2}{1-\alpha^2}\Vert y_0\Vert_X^2 .
$$
Let us prove that $\Vert y(t;y_0,u)\Vert_X$ decreases exponentially. The argument is standard. Taking $t\geq 0$ such that $kT\leq t<(k+1)T$ for some $k\in\N$, we have
$$
y(t;y_0,u) = S(t-kT)y(kT) + \int_{kT}^t S(t-s)Bu_k(s)\, ds .
$$
The semigroup $(S(t))_{t\geq 0}$ satisfies $\Vert S(t)\Vert_{L(X)}\leq M$ for every $t\in[0,T]$ for some $M\geq 1$. Therefore
$$
\Vert y(t;y_0,u)\Vert_X\leq M\left(1+\Vert B\Vert_{L(U,X)}\sqrt{T}\right) C \Vert y_k\Vert_X
\leq \frac{M}{\alpha}\left(1+\Vert B\Vert_{L(U,X)}\sqrt{T}\right) C \alpha^{k+1} \Vert y_0\Vert_X
$$
and since $\alpha^{k+1}=e^{(k+1)T\frac{\ln\alpha}{T}}\leq e^{\frac{\ln\alpha}{T}t}$ we infer that
$$
\Vert y(t;y_0,u)\Vert_X\leq \frac{M}{\alpha}\left(1+\Vert B\Vert_{L(U,X)}\sqrt{T}\right) C \Vert y_0\Vert_X e^{\frac{\ln\alpha}{T}t} \qquad\forall t>0.
$$
We have therefore found a control $u\in L^2(0,+\infty;U)$ such that
\begin{equation}\label{Jineq}
\int_0^{+\infty} \left( \Vert u(t)\Vert_U^2+\Vert y(t;y_0,u)\Vert_X^2\right)\, dt < +\infty .
\end{equation}
Hence, by the classical Riccati theory (see \cite[Theorem 4.3 page 240]{Z}), the control system \eqref{contsyst} is exponentially stabilizable.

The statement concerning the best stabilization rate is obvious by inspecting the above argument.
\end{proof}

\begin{remark}\label{remRiccati}
To prove Theorem \ref{thm_stab_extended} in Section \ref{sec:furthercomments}, we note that all arguments in the above proof still work (before the final step where Riccati theory is invoked) for an unbounded admissible control operator $B$, because the operators $L_t=\int_0^t S(t-s)Bu(s)\, ds$ are bounded in $X$ and we can write
\begin{multline*}
\left\Vert \int_{kT}^t S(t-s)Bu_k(s)\, ds\right\Vert_X = \left\Vert \int_0^{t-kT} S(t-kT-s)Bu_k(s+kT)\, ds\right\Vert_X \\
= \Vert L_{t-kT}u_k(kT+\cdot)\Vert_X \leq C \Vert u_k\Vert_{L^2(kT,(k+1)T;U)}
\end{multline*}
and the rest is unchanged: we find $u$ satisfying \eqref{Jineq}.

But for the final step of the proof, Riccati theory is classical and quite easy for bounded control operators but is much more intricate for unbounded control operators. Riccati theory is well established, in a general framework, for analytic semigroups (see \cite{LT1}).
\end{remark}

Theorem \ref{thm_stab} follows from Lemmas \ref{lem_interpret_muT} and \ref{lem_expstab}.

\appendix

\section{Appendix: exact null controllability implies complete stabilizability}\label{app:exactnull_implies_complete_stab}
We have the following result.

\begin{proposition}\label{prop_exactnull_implies_complete_stab}
If the control system \eqref{contsyst} is exactly null controllable in some time $T>0$ then it is completely stabilizable.
\end{proposition}

\begin{proof}
We first note that $(A,B)$ is exactly controllable in time $T$ if and only if $(A+\omega\mathrm{id},B)$ is exactly controllable in time $T$, for any $\omega\in\R$. This follows straightforwardly by using the equivalence in terms of observability inequality and the fact that $S_{A+\omega\mathrm{id}}(t) = e^{\omega t}S_A(t)$ (with obvious notations).

Now, let $\omega>0$ be arbitrary. Since $(A+\omega\mathrm{id},B)$ is exactly controllable in time $T$, there exists $K_\omega\in L(X,U)$ such that $A+\omega\mathrm{id}+BK_\omega$ generates the exponentially stable semigroup $S_{A+\omega\mathrm{id}+BK_\omega}(t) = e^{\omega t} S_{A+BK_\omega}(t)$, and thus $\Vert S_{A+BK_\omega}(t)\Vert_{L(X)}\leq Me^{-\omega t}$ for some $M\geq 1$. The result follows.
\end{proof}

Surprisingly, we have not found this result in the existing literature. What can be found is that exact null controllability in some time $T$ implies expoential stabilizability (see, e.g., \cite{Z}) and that, when $(S(t))_{t\geq 0}$ is a group, exact null controllability in some time $T$ implies complete stabilizability (see \cite{Slemrod}; and the converse is true: see \cite{Z}). Here, in Proposition \ref{prop_exactnull_implies_complete_stab}, we do not require that $(S(t))_{t\geq 0}$ is a group.

\end{document}